\documentclass[12pt]{article}
\usepackage{amsmath,amsthm}
\usepackage{amssymb}
\usepackage{hyperref}
\usepackage{amsfonts}
\usepackage{color, url}

\begin{document}

\title{Convolution identities for Tribonacci numbers with symmetric formulae}
\author{
Takao Komatsu 
\, and
{Rusen Li}   \\
\small School of Mathematics and Statistics\\
\small Wuhan University\\
\small Wuhan 430072 China\\
\small \texttt{komatsu@whu.edu.cn} \\
\small  \texttt{limanjiashe@whu.edu.cn}
}

\date{
}

\maketitle

\def\stf#1#2{\left[#1\atop#2\right]}
\def\sts#1#2{\left\{#1\atop#2\right\}}
\def\fl#1{\left\lfloor#1\right\rfloor}
\def\cl#1{\left\lceil#1\right\rceil}

\newtheorem{theorem}{Theorem}
\newtheorem{Prop}{Proposition}
\newtheorem{Cor}{Corollary}
\newtheorem{Lem}{Lemma}

\begin{abstract}  
Many kinds of convolution identities have been considered about several numbers, including Bernoulli, Euler, Genocchi, Cauchy, Stirling, and Fibonacci numbers. The well-known basic result about Bernoulli numbers is due to Euler. The convolution identities have been studied. 

In this paper, by using symmetric formulas, we give convolution identities for Tribonacci numbers, in particular, of higher-order. Convolution identities of Fibonacci numbers or Lucas numbers can be expressed in the form of linear combinations of Fibonacci numbers and Lucas numbers only. Fibonacci numbers and Lucas numbers are in pairs with different values. Convolution identities of Tribonacci numbers can be expressed in the linear combination of several Tribonacci numbers with different values.  \\
\noindent 
{\bf AMS 2010 Subject Classifications}: Primary 11B39; Secondary 11B37, 05A15, 05A19\\
\noindent 
{\bf Key words}: Convolution identities, symmetric formulas, cubic equations, Tribonacci numbers 
\end{abstract}

\section{Introduction}

Convolution identities for various types of numbers (or polynomials) have been studied, with or without binomial (or multinomial) coefficients, including Bernoulli, Euler, Genocchi, Cauchy, Stirling, and Fibonacci numbers (\cite{AD1,AD2,AD3,Komatsu2015,Komatsu2016,KMP,KS2016}). One typical formula is due to Euler, given by
$$
\sum_{k=0}^n\binom{n}{k}\mathcal B_k\mathcal B_{n-k}=-n\mathcal B_{n-1}-(n-1)\mathcal B_n\quad(n\ge 0)\,,
$$
where $\mathcal B_n$ are Bernoulli numbers, defined by
$$
\frac{x}{e^x-1}=\sum_{n=0}^\infty\mathcal B_n\frac{x^n}{n!}\quad(|x|<2\pi)\,.
$$
In \cite{KP}, Panda et al. several kinds of the sums of product of two balancing numbers are given.  As an application, the sums of the products of two Fibonacci (and Lucas) numbers 
$$
\sum_{m=0}^n F_{k m+r}F_{k(n-m)+r}\quad\hbox{and}\quad \sum_{m=0}^n L_{k m+r}L_{k(n-m)+r} 
$$ 
are given,  
where $k$ and $r$ are fixed integers with $k>r\ge 0$.  
The case without binomial (multinomial) coefficients about Fibonacci numbers is discussed in \cite{KMP}.    
In \cite{KR}, the convolution identities for more general Fibonacci-type numbers $u_n$, satisfying the recurrence relations $u_n=a u_{n-1}+b u_{n-2}$, have been studied.  By considering the roots of the quadratic equation $x^2-a x-b=0$,
$$
\sum_{k_1+\cdots+k_r=n\atop k_1,\dots,k_r\ge 0}\binom{n}{k_1,\dots,k_r}u_{k_1}\cdots u_{k_r}
$$
can be expressed in the linear combination of $u_1$, $\dots$, $u_n$, where
$$
\binom{n}{k_1,\dots,k_r}=\frac{n!}{k_1!\cdots k_r!}
$$
denotes the multinomial coefficient.  

However, the situation becomes more difficult for the numbers related to the higher-order equation.  Nevertheless, as the cubic equation can be solvable, the numbers satisfying the four-term recurrence relation have the possibility to have the convolution identities.

{\it Tribonacci numbers} $T_n$ are defined by the recurrence relation
\begin{equation}
T_n=T_{n-1}+T_{n-2}+T_{n-3}\quad(n\ge 3)\quad\hbox{with}\quad T_0=0,~T_1=T_2=1
\label{def:tribo}
\end{equation}
and their sequence is given by
$$
\{T_n\}_{n\ge 0}= 0, 1, 1, 2, 4, 7, 13, 24, 44, 81, 149, \dots
$$
({\it. Cf.} \cite[A000073]{oeis}). 

The generating function without factorials is given by
\begin{equation}
T(x):=\frac{x}{1-x-x^2-x^3}=\sum_{n=0}^\infty T_n x^n
\label{gen:nofacto}
\end{equation}
because of the recurrence relation (\ref{def:tribo}).

On the other hand, the generating function with binomial coefficients is given by
\begin{equation}
t(x):=c_1 e^{\alpha x}+c_2 e^{\beta x}+c_3 e^{\gamma x}=\sum_{n=0}^\infty T_n\frac{x^n}{n!}\,,
\label{gen:facto}
\end{equation}
where $\alpha$, $\beta$ and $\gamma$ are the roots of $x^3-x^2-x-1=0$ and given by
\begin{align*}
\alpha&=\frac{\sqrt[3]{19+3\sqrt{33}}+\sqrt[3]{19-3\sqrt{33}}+1}{3}=1.839286755\,,\\
\beta,\gamma&=\frac{2-(1\pm\sqrt{-3})\sqrt[3]{19-3\sqrt{33}}-(1\mp\sqrt{-3})\sqrt[3]{19+3\sqrt{33}}}{6}\\
&=-0.4196433776\pm0.6062907292\sqrt{-1}
\end{align*}
and
\begin{align*}
c_1:&=\frac{\alpha}{(\alpha-\beta)(\alpha-\gamma)}=\frac{1}{-\alpha^2+4\alpha-1}=0.3362281170,\\
c_2:&=\frac{\beta}{(\beta-\alpha)(\beta-\gamma)}=\frac{1}{-\beta^2+4\beta-1}\\
&=-0.1681140585 - 0.1983241401\sqrt{-1},\\
c_3:&=\frac{\gamma}{(\gamma-\alpha)(\gamma-\beta)}=\frac{1}{-\gamma^2+4\gamma-1}\\
&=-0.1681140585 + 0.1983241401\sqrt{-1}
\end{align*}
(see e.g., \cite{K2011})
Notice that
\begin{align*}
c_1+c_2+c_3&=0\,,\\
c_1\alpha+c_2\beta+c_3\gamma&=1\,,\\
c_1\alpha^2+c_2\beta^2+c_3\gamma^2&=1\,,
\end{align*}
because $T_n$ has a Binet-type formula (see e.g., \cite{Kiric2008}):
$$
T_n=c_1\alpha^n+c_2\beta^n+c_3\gamma^n\quad(n\ge 0)\,.
$$

In this paper, by using symmetric formulas, we give convolution identities, in particular, with binomial coefficients of higher-order. Convolution identities for Fibonacci numbers can be expressed in the linear combinations of Fibonacci and Lucas numbers only.  Fibonacci numbers and Lucas numbers are in pairs with different values.   Those of Tribonacci numbers can be expressed in the linear combinations of various Tribonacci numbers with different initial values.

\section{Convolution identities without binomial coefficients}

In \cite{Komatsu2018}, convolution identites without binomial coefficients are studied, and the following results are obtained. 

\begin{Prop}
For $n\ge 3$, we have
$$
\sum_{k=0}^{n-3}T_k(T_{n-k}+T_{n-k-2}+2 T_{n-k-3})=(n-2)T_{n-1}-T_{n-2}\,.
$$
\label{th:nobinom}
\end{Prop}

\begin{Prop}
For $n\ge 2$,
\begin{multline*}
\sum_{k=0}^n T_k T_{n-k}\\
=\sum_{l=1}^{n-1}\left(\sum_{i=0}^{\fl{\frac{n-l-1}{3}}}2^{i-1}\left((-1)^{\frac{n-l-i-1}{2}}+(-1)^{\frac{3(n-l-i-1)}{2}}\right)\binom{\frac{n-l-i-1}{2}}{i}\right)l T_l\,.
\end{multline*}
\end{Prop} 

It is not easy for 
the cases of the sum of products of more than two Tribonacci numbers.  
By (\ref{gen:nofacto}), we have
$$
T'(x)=\frac{1+x^2+2 x^3}{(1-x-x^2-x^3)^2}
$$
and 
$$
T''(x)=\frac{2+6 x +12 x^2+6 x^4+6 x^5}{(1-x-x^2-x^3)^3}\,.
$$
Hence,
\begin{equation}
(2+6 x +12 x^2+6 x^4+6 x^5)T(x)^3=x^3 T''(x)\,.
\label{eq:202}
\end{equation}

\begin{align*}
x^3 T''(x)=\sum_{n=2}^\infty n(n-1)T_{n}x^{n+1}
=\sum_{n=3}^\infty(n-1)(n-2)T_{n-1}x^n\,.
\end{align*}

\begin{align*}\
&(6 x^5+6 x^4+12 x^2+6 x +2)T(x)^3\\
&=\sum_{n=5}^\infty 6 \sum_{k_1+k_2+k_3=n-5\atop k_1,k_2,k_3\ge 0}T_{k_1}T_{k_2}T_{k_3}x^n
+ \sum_{n=4}^\infty 6 \sum_{k_1+k_2+k_3=n-4\atop k_1,k_2,k_3\ge 0}T_{k_1}T_{k_2}T_{k_3}x^n\\
&\quad +\sum_{n=2}^\infty 12 \sum_{k_1+k_2+k_3=n-2\atop k_1,k_2,k_3\ge 0}T_{k_1}T_{k_2}T_{k_3}x^n
+ \sum_{n=1}^\infty 6 \sum_{k_1+k_2+k_3=n-1\atop k_1,k_2,k_3\ge 0}T_{k_1}T_{k_2}T_{k_3}x^n\\
&\quad + \sum_{n=0}^\infty 2 \sum_{k_1+k_2+k_3=n\atop k_1,k_2,k_3\ge 0}T_{k_1}T_{k_2}T_{k_3}x^n.
\end{align*}

Therefore, we get the following result.

\begin{theorem}
For $n\ge 5$, we have
\begin{align*}
&(n-1)(n-2)T_{n-1}\\
&=6 \sum_{k_1+k_2+k_3=n-5\atop k_1,k_2,k_3\ge 0}T_{k_1}T_{k_2}T_{k_3}
+6 \sum_{k_1+k_2+k_3=n-4\atop k_1,k_2,k_3\ge 0}T_{k_1}T_{k_2}T_{k_3}\\
&\quad +12 \sum_{k_1+k_2+k_3=n-2\atop k_1,k_2,k_3\ge 0}T_{k_1}T_{k_2}T_{k_3}
+6 \sum_{k_1+k_2+k_3=n-1\atop k_1,k_2,k_3\ge 0}T_{k_1}T_{k_2}T_{k_3}\\
&\quad +2 \sum_{k_1+k_2+k_3=n\atop k_1,k_2,k_3\ge 0}T_{k_1}T_{k_2}T_{k_3}\,.
\end{align*}
\label{th:nobinom1}
\end{theorem}

\section{Convolution identities with binomial coefficients}

For convenience, we shall introduce modified Tribonacci numbers $T_n^{(s_0,s_1,s_2)}$, satisfying the recurrence relation
$$
T_n^{(s_0,s_1,s_2)}=T_{n-1}^{(s_0,s_1,s_2)}+T_{n-2}^{(s_0,s_1,s_2)}+T_{n-3}^{(s_0,s_1,s_2)}\quad(n\ge 3)
$$
with given initial values $T_0^{(s_0,s_1,s_2)}=s_0$, $T_1^{(s_0,s_1,s_2)}=s_1$ and $T_2^{(s_0,s_1,s_2)}=s_2$.  Hence, $T_n=T_n^{(0,1,1)}$ are ordinary Tribonacci numbers.

In \cite{Komatsu2018},  we show the following two lemmata.

\begin{Lem}
We have
$$
c_1^2 e^{\alpha x}+c_2^2 e^{\beta x}+c_3^2 e^{\gamma x}=\frac{1}{22}\sum_{n=0}^\infty T_n^{(2,3,10)}\frac{x^n}{n!}\,.
$$
\label{c^2}
\end{Lem}

\begin{Lem}
We have
$$
c_2 c_3 e^{\alpha x}+c_3 c_1 e^{\beta x}+c_1 c_2 e^{\gamma x}=\frac{1}{22}\sum_{n=0}^\infty T_n^{(-1,2,7)}\frac{x^n}{n!}\,.
$$
\label{cc}
\end{Lem}

By using Lemma \ref{c^2} and Lemma \ref{cc} together with the fact  
\begin{equation}  
c_1 c_2+c_2 c_3+c_3 c_1=-\frac{1}{22}\,, 
\label{cc+cc+cc} 
\end{equation}  
we can show the convolution identity of two Tribonacci numbers (\cite[Theorem 1]{Komatsu2018}).  

\begin{Prop}
For $n\ge 0$,
$$
\sum_{k=0}^n\binom{n}{k}T_k T_{n-k}=\frac{1}{22}\left(2^n T_n^{(2,3,10)}+2\sum_{k=0}^n\binom{n}{k}(-1)^k T_k^{(-1,2,7)}\right)\,.
$$
\label{tri2}
\end{Prop}

\section{Symmetric formulae}

Before giving more convolution identities, we shall give some basic algebraic identities in symmetric forms. It is not so difficult to determine the relations among coefficients. So, we list three results without proof.  

\begin{Lem}\label{alg-3}
The following equality holds:
\begin{align*}
&(a+b+c)^3
=A (a^3+b^3+c^3) + B abc \\
&\quad + C (a^2+b^2+c^2)(a+b+c) + D(ab + bc + ca)(a+b+c),
\end{align*}
where $A=D-2$, $B=-3D+6$ and $C=-D+3$.
\end{Lem}

\begin{Lem}
The following equality holds:
\begin{align*}
&(a+b+c)^4 \\
&=A(a^4+b^4+c^4)+C(a^3+b^3+c^3)(a+b+c)+D(a^2+b^2+c^2)^2  \\
&\quad +E(a^2+b^2+c^2)(ab + bc + ca)+F(ab + bc + ca)^2\\
&\quad +G(a^2+b^2+c^2)(a+b+c)^2\\
&\quad +H(ab + bc + ca)(a+b+c)^2+Iabc(a+b+c),
\end{align*}
where $A=-D+E+G+H-3$, $C=-E-2G-H+4$,
$F=-2D-2G-2H+6$ and $I=4D-E+2G-H$.
\label{alg-4}
\end{Lem}

\begin{Lem}
The following equality holds:
\begin{align*}
&(a+b+c)^5 \\
&=A(a^5+b^5+c^5)+Babc(ab+bc+ca)+Cabc(a^2+b^2+c^2)\\
&\quad +Dabc(a+b+c)^2+E(a^4+b^4+c^4)(a+b+c)\\
&\quad +H(a^3+b^3+c^3)(a^2+b^2+c^2)\\
&\quad +I(a^3+b^3+c^3)(ab+bc+ca)+L(a^3+b^3+c^3)(a+b+c)^2\\
&\quad +N(a^2+b^2+c^2)^2(a+b+c)+P(ab + bc + ca)^2(a+b+c)\\
&\quad +Q(a^2+b^2+c^2)(ab + bc + ca)(a+b+c)+R(a^2+b^2+c^2)(a+b+c)^3\\
&\quad +S(ab + bc + ca)(a+b+c)^3,
\end{align*}
where $A=I+2L+2N+P+2Q+6R+4S-14$, $B=-2D-2N-5P-2Q-6R-12S+30$,
$C=-D-I-2L-2P-3Q-6R-7S+20$, $E=-I-2L-N-Q-3R-S+5$ and $H=-L-2N-P-Q-4R-3S+10$.
\label{alg-5}
\end{Lem}
\bigskip

Now, we shall consider the sum of the products of three Tribonacci numbers.
We need two more supplementary results.

\begin{Lem}
We have
$$
c_1 c_2 c_3=\frac{1}{44}\,.
$$
\label{ccc}
\end{Lem}
\begin{proof}
Since
\begin{align*}
(\alpha-\beta)(\alpha-\gamma)&=\alpha^2-\alpha(\beta+\gamma)+\beta\gamma\\
&=\alpha^2-\alpha(1-\alpha)+\alpha^2-\alpha-1\\
&=3\alpha^2-2\alpha-1=(3\alpha+1)(\alpha-1)\,,
\end{align*}
we get
\begin{align*}
&(\alpha-\beta)^2(\beta-\gamma)^2(\gamma-\alpha)^2\\
&=-(3\alpha+1)(\alpha-1)(3\beta+1)(\beta-1)(3\gamma+1)(\gamma-1)\\
&=-\bigl(27\alpha\beta\gamma+9(\alpha\beta+\beta\gamma+\gamma\alpha)+3(\alpha+\beta+\gamma)+1\bigr)\\
&\quad\times\bigl(\alpha\beta\gamma-(\alpha\beta+\beta\gamma+\gamma\alpha)+(\alpha+\beta+\gamma)-1\bigr)\\
&=-22\cdot 2=-44\,.
\end{align*}
Thus,
$$
c_1 c_2 c_3=\frac{\alpha\beta\gamma}{-(\alpha-\beta)^2(\beta-\gamma)^2(\gamma-\alpha)^2}=\frac{1}{44}\,.
$$
\end{proof}

\begin{Lem}
We have
$$
c_1^3 e^{\alpha x}+c_2^3 e^{\beta x}+c_3^3 e^{\gamma x}=\frac{1}{44}\sum_{n=0}^\infty T_n^{(3,3,5)}\frac{x^n}{n!}\,.
$$
\label{c^3}
\end{Lem}
\begin{proof}
In the proof of Lemma \ref{c^2}, we put $s_0=3$, $s_1=3$ and $s_2=5$, instead.
By $(\alpha-\beta)(\alpha-\gamma)=3\alpha^2-2\alpha-1$ and $\beta\gamma=\alpha^2-\alpha-1$, we get
$$
(\alpha-\beta)^2(\alpha-\gamma)^2(-3\beta\gamma+3(\beta+\gamma)-5)=-44\alpha^3\,.
$$
Thus, we obtain that
$$
d_1=\frac{-3\beta\gamma+3(\beta+\gamma)-5}{(\alpha-\beta)(\gamma-\alpha)}=\frac{44\alpha^3}{(\alpha-\beta)^3(\alpha-\gamma)^3}=44 c_1^3\,.
$$
Similarly, we obtain that $d_2=44 c_2^3$ and $d_3=44 c_3^2$.
\end{proof}

\section{Convolution identities for three, four and five Tribonacci numbers}

By using Lemmata \ref{c^2}, \ref{cc},  \ref{alg-3}, \ref{ccc} and \ref{c^3}, we get the following result. Notice that $T_0=0$, so that $k_1,k_2,k_3$ do not have to include $0$.

\begin{theorem}
For $n\ge 0$,
\begin{align*}
&\sum_{k_1+k_2+k_3=n\atop k_1,k_2,k_3\ge 1}\binom{n}{k_1,k_2,k_3}T_{k_1}T_{k_2}T_{k_3}\\
&=\frac{A}{44}3^nT_{n}^{(3,3,5)}+\frac{B}{44}+\frac{C}{22}\sum_{k=0}^n\binom{n}{k}2^{n-k}T_{n-k}^{(2,3,10)}T_k\\
&+\frac{D}{22}\sum_{k_1+k_2+k_3=n\atop k_1,k_2,k_3\ge0}\binom{n}{k_1,k_2,k_3}(-1)^{k_1}T_{k_1}^{(-1,2,7)}T_{k_2}\,,
\end{align*}
\label{tri3}
where $A=D-2$, $B=-3D+6$ and $C=-D+3$.
\end{theorem}

\noindent
{\it Remark.}
If we take $D=0$, we have for $n\ge 0$,
\begin{align*}
&\sum_{k_1+k_2+k_3=n\atop k_1,k_2,k_3\ge 1}\binom{n}{k_1,k_2,k_3}T_{k_1}T_{k_2}T_{k_3}\\
&=\frac{1}{22}\left(3\sum_{k=0}^n\binom{n}{k}2^{n-k}T_{n-k}^{(2,3,10)}T_k-3^n T_n^{(3,3,5)}+3\right)\,.
\end{align*}

\begin{proof}[Proof of Theorem \ref{tri3}]
First, by Lemmata \ref{c^2}, \ref{cc},  \ref{alg-3}, \ref{ccc} and \ref{c^3}, we have
\begin{eqnarray*}
&&(c_1 e^{\alpha x}+c_2 e^{\beta x}+c_3 e^{\gamma x})^3\\
&=&A(c_1^3 e^{3\alpha x}+c_2^3 e^{3\beta x}+c_3^3 e^{3\gamma x})+B c_1 c_2 c_3 e^{(\alpha+\beta+\gamma)x}\\
&&+C(c_1^2 e^{2\alpha x}+c_2^2 e^{2\beta x}+c_3^2 e^{2\gamma x})(c_1 e^{\alpha x}+c_2 e^{\beta x}+c_3 e^{\gamma x})\\
&&+D(c_1c_2e^{(\alpha+\beta) x}+c_2c_3e^{(\beta+\gamma)x}+c_3c_1e^{(\gamma+\alpha)x})(c_1 e^{\alpha x}+c_2 e^{\beta x}+c_3 e^{\gamma x})\\
&=&\frac{A}{44}\sum_{n=0}^\infty T_{n}^{(3,3,5)}\frac{(3 x)^n}{n!}+\frac{B}{44}\sum_{n=0}^\infty \frac{x^n}{n!}\\
&&+\frac{C}{22}\sum_{n=0}^\infty \sum_{k=0}^n\binom{n}{k}2^{n-k}T_{n-k}^{(2,3,10)}T_k \frac{x^n}{n!}\\
&&+\frac{D}{22}\sum_{n=0}^\infty \sum_{k_1+k_2+k_3=n\atop k_1,k_2,k_3\ge0}\binom{n}{k_1,k_2,k_3}T_{k_1}^{(-1,2,7)}(-1)^{k_1}T_{k_2}\frac{x^n}{n!}\,.
\end{eqnarray*}
On the other hand,
$$
\left(\sum_{n=0}^\infty T_n\frac{x^n}{n!}\right)^3=\sum_{k_1+k_2+k_3=n\atop k_1,k_2,k_3\ge 1}\binom{n}{k_1,k_2,k_3}T_{k_1}T_{k_2}T_{k_3}\frac{x^n}{n!}\,.
$$
Comparing the coefficients on both sides, we get the desired result.
\end{proof}
\bigskip

Next, we shall consider the sum of the products of four Tribonacci numbers.
We need the following supplementary result.  The proof is similar to that of Lemma \ref{c^3} and omitted.

\begin{Lem}
We have
$$
c_1^4 e^{\alpha x}+c_2^4 e^{\beta x}+c_3^4 e^{\gamma x}=\frac{1}{484}\sum_{n=0}^\infty T_n^{(2,14,21)}\frac{x^n}{n!}\,.
$$
\label{c^4}
\end{Lem}

\begin{theorem}
For $n\ge 0$,
\begin{align*}
&\sum_{k_1+k_2+k_3+k_4=n\atop k_1,k_2,k_3,k_4\ge 1}\binom{n}{k_1,k_2,k_3,k_4}T_{k_1}T_{k_2}T_{k_3}T_{k_4}\\
&=\frac{A}{484}4^n T_n^{(2,14,21)}+\frac{C}{44}\sum_{k=0}^n\binom{n}{k}3^{n-k}T_{n-k}^{(3,3,5)}T_k\\
&\quad +\frac{D}{484}\sum_{k=0}^n\binom{n}{k}2^{n}T_{n-k}^{(2,3,10)}T_k ^{(2,3,10)}\\
&\quad +\frac{E}{484}\sum_{k_1+k_2+k_3=n\atop k_1,k_2,k_3\ge0}\binom{n}{k_1,k_2,k_3}(-1)^{k_1}T_{k_1}^{(-1,2,7)}T_{k_2}^{(2,3,10)}2^{k_2}\\
&\quad +\frac{F}{484}\sum_{k_1+k_2+k_3+k_4=n\atop k_1,k_2,k_3,k_4\ge 0}\binom{n}{k_1,k_2,k_3,k_4}(-1)^{k_1+k_2}T_{k_1}^{(-1,2,7)}T_{k_2}^{(-1,2,7)}\\
&\quad +\frac{G}{22}\sum_{k_1+k_2+k_3=n\atop k_1,k_2,k_3\ge0}\binom{n}{k_1,k_2,k_3}2^{k_1}T_{k_1}^{(2,3,10)}T_{k_2}T_{k_3}\\
&\quad +\frac{H}{22}\sum_{k_1+k_2+k_3+k_4=n\atop k_1,k_2,k_3,k_4\ge 0}\binom{n}{k_1,k_2,k_3,k_4}(-1)^{k_1}T_{k_1}^{(-1,2,7)}T_{k_2}T_{k_3}\\
&\quad +\frac{I}{44}\sum_{k=0}^n\binom{n}{k}T_k\,,
\end{align*}
where $A=-D+E+G+H-3$, $C=-E-2G-H+4$, $F=-2D-2G-2H+6$ and $I=4D-E+2G-H$.
\label{tri4}
\end{theorem}

\noindent
{\it Remark.}
If $E=F=G=H=0$, then by $A=-6$, $C=4$, $D=3$ and $I=12$, we have for $n\ge 0$,
\begin{align*}
&\sum_{k_1+k_2+k_3+k_4=n\atop k_1,k_2,k_3,k_4\ge 1}\binom{n}{k_1,k_2,k_3,k_4}T_{k_1}T_{k_2}T_{k_3}T_{k_4}\\
&=\frac{3}{484}\left(-2\cdot 4^n T_n^{(2,14,21)}+\sum_{k=0}^n\binom{n}{k}2^{n}T_{n-k}^{(2,3,10)}T_k ^{(2,3,10)}\right)\\
&\quad +\frac{1}{11}\sum_{k=0}^n\binom{n}{k}T_k\left(3+3^{n-k}T_{n-k}^{(3,3,5)}\right)\,.
\end{align*}

\begin{proof}[Proof of Theorem \ref{tri4}]
By Lemmata  \ref{c^2}, \ref{cc},  \ref{alg-4}, \ref{ccc}, \ref{c^3} and \ref{c^4}, we have
\begin{align*}
&(c_1 e^{\alpha x}+c_2 e^{\beta x}+c_3 e^{\gamma x})^4\\
&=A(c_1^4 e^{4\alpha x}+c_2^4 e^{4\beta x}+c_3^4 e^{4\gamma x})\\
&\quad +C(c_1^3 e^{3\alpha x}+c_2^3 e^{3\beta x}+c_3^3 e^{3\gamma x})(c_1 e^{\alpha x}+c_2 e^{\beta x}+c_3 e^{\gamma x})\\
&\quad +D(c_1^2 e^{2\alpha x}+c_2^2 e^{2\beta x}+c_3^2 e^{2\gamma x})^2\\
&\quad +E(c_1^2 e^{2\alpha x}+c_2^2 e^{2\beta x}+c_3^2 e^{2\gamma x})(c_1 c_2 e^{(\alpha+\beta)x}+c_2 c_3 e^{(\beta+\gamma)x}+c_3 c_1 e^{(\gamma+\alpha)x})\\
&\quad +F(c_1 c_2 e^{(\alpha+\beta)x}+c_2 c_3 e^{(\beta+\gamma)x}+c_3 c_1 e^{(\gamma+\alpha)x})^2\\
&\quad +G(c_1^2 e^{2\alpha x}+c_2^2 e^{2\beta x}+c_3^2 e^{2\gamma x})(c_1 e^{\alpha x}+c_2 e^{\beta x}+c_3 e^{\gamma x})^2\\
&\quad +H(c_1 c_2 e^{(\alpha+\beta)x}+c_2 c_3 e^{(\beta+\gamma)x}+c_3 c_1 e^{(\gamma+\alpha)x})(c_1 e^{\alpha x}+c_2 e^{\beta x}+c_3 e^{\gamma x})^2\\
&\quad +I(c_1 e^{\alpha x}+c_2 e^{\beta x}+c_3 e^{\gamma x})c_1 c_2 c_3 e^{(\alpha+\beta+\gamma)x}\\
&=\frac{A}{484}\sum_{n=0}^\infty 4^n T_n^{(2,14,21)}\frac{x^n}{n!}+\frac{C}{44}\sum_{n=0}^\infty \sum_{k=0}^n\binom{n}{k}3^{n-k}T_{n-k}^{(3,3,5)}T_k\frac{x^n}{n!}\\
&\quad +\frac{D}{484}\sum_{n=0}^\infty \sum_{k=0}^n\binom{n}{k}2^{n}T_{n-k}^{(2,3,10)}T_k ^{(2,3,10)}\frac{x^n}{n!}\\
&\quad +\frac{E}{484}\sum_{n=0}^\infty \sum_{k_1+k_2+k_3=n\atop k_1,k_2,k_3\ge0}\binom{n}{k_1,k_2,k_3}T_{k_1}^{(-1,2,7)}(-1)^{k_1}T_{k_2}^{(2,3,10)}2^{k_2}\frac{x^n}{n!}\\
&\quad +\frac{F}{484}\sum_{n=0}^\infty \sum_{k_1+k_2+k_3+k_4=n\atop k_1,k_2,k_3,k_4\ge 0}\binom{n}{k_1,k_2,k_3,k_4}T_{k_1}^{(-1,2,7)}(-1)^{k_1}T_{k_2}^{(-1,2,7)}(-1)^{k_2}\frac{x^n}{n!}\\
&\quad +\frac{G}{22}\sum_{n=0}^\infty \sum_{k_1+k_2+k_3=n\atop k_1,k_2,k_3\ge0}\binom{n}{k_1,k_2,k_3}T_{k_1}^{(2,3,10)}2^{k_1}T_{k_2}T_{k_3}\frac{x^n}{n!}\\
&\quad +\frac{H}{22}\sum_{n=0}^\infty \sum_{k_1+k_2+k_3+k_4=n\atop k_1,k_2,k_3,k_4\ge 0}\binom{n}{k_1,k_2,k_3,k_4}T_{k_1}^{(-1,2,7)}(-1)^{k_1}T_{k_2}T_{k_3}\frac{x^n}{n!}\\
&\quad +\frac{I}{44}\sum_{n=0}^\infty \sum_{k=0}^n\binom{n}{k}T_k\frac{x^n}{n!}\,.
\end{align*}
On the other hand,
$$
\left(\sum_{n=0}^\infty T_n\frac{x^n}{n!}\right)^4=\sum_{k_1+k_2+k_3+k_4=n\atop k_1,k_2,k_3,k_4\ge 1}\binom{n}{k_1,k_2,k_3,k_4}T_{k_1}T_{k_2}T_{k_3}T_{k_4}\frac{x^n}{n!}\,.
$$
Comparing the coefficients on both sides, we get the desired result.
\end{proof}
\bigskip

Next, we shall consider the sum of the products of five Tribonacci numbers.
We need the following supplementary result.  The proof is similar to that of Lemma \ref{c^3} and omitted.

\begin{Lem}
$$
c_1^5 e^{\alpha x}+c_2^5 e^{\beta x}+c_3^5 e^{\gamma x}=\frac{1}{968}\sum_{n=0}^\infty T_n^{(5,6,15)}\frac{x^n}{n!}.
$$
\label{c^5}
\end{Lem}


Similarly, by Lemmata  \ref{c^2}, \ref{cc}, \ref{alg-5}, \ref{ccc},  \ref{c^3}, \ref{c^4} and \ref{c^5}, we have the following result about the sum of the products of five Tribonacci numbers.

\begin{theorem}
For $n\ge 0$,
\begin{align*}
&\sum_{k_1+\cdots+k_5=n\atop k_1,\dots,k_5\ge 1}\binom{n}{k_1,\dots,k_5}T_{k_1}\cdots T_{k_5}\\
&=\frac{A}{968}5^n T_n^{(5,6,15)}+\frac{B}{968}\sum_{k_1+k_2+k_3=n\atop k_1,k_2,k_3\ge0}\binom{n}{k_1,k_2,k_3}T_{k_1}^{(-1,2,7)}(-1)^{k_1}\\
&\quad +\frac{C}{968}\sum_{k=0}^n\binom{n}{k}2^{k}T_{k}^{(2,3,10)}
+\frac{D}{44}\sum_{k_1+k_2+k_3=n\atop k_1,k_2,k_3\ge0}\binom{n}{k_1,k_2,k_3}T_{k_1}T_{k_2}\\
&\quad +\frac{E}{484}\sum_{k=0}^n\binom{n}{k}4^{n-k}T_{n-k}^{(2,14,21)}T_k
+\frac{H}{968}\sum_{k=0}^n\binom{n}{k}3^{n-k}2^k T_{n-k}^{(3,3,5)}T_k^{(2,3,10)}\\
&\quad +\frac{I}{968}\sum_{k_1+k_2+k_3=n\atop k_1,k_2,k_3\ge0}\binom{n}{k_1,k_2,k_3}3^{k_1}(-1)^{k_2}T_{k_1}^{(3,3,5)}T_{k_2}^{(-1,2,7)}\\
&\quad +\frac{L}{44}\sum_{k_1+k_2+k_3=n\atop k_1,k_2,k_3\ge0}\binom{n}{k_1,k_2,k_3}3^{k_1}T_{k_1}^{(3,3,5)}T_{k_2}T_{k_3}\\
&\quad +\frac{N}{484}\sum_{k_1+k_2+k_3=n\atop k_1,k_2,k_3\ge0}\binom{n}{k_1,k_2,k_3}2^{k_1}T_{k_1}^{(2,3,10)}2^{k_2}T_{k_2}^{(2,3,10)}T_{k_3}\\
&\quad +\frac{P}{484}\sum_{k_1+\cdots+k_5=n\atop k_1,\dots,k_5\ge 0}\binom{n}{k_1,\dots,k_5}(-1)^{k_1+k_2}T_{k_1}^{(-1,2,7)}T_{k_2}^{(-1,2,7)}T_{k_3}\\
&\quad +\frac{Q}{484}\sum_{k_1+k_2+k_3+k_4=n\atop k_1,k_2,k_3,k_4\ge 0}\binom{n}{k_1,k_2,k_3,k_4}2^{k_1}(-1)^{k_2}T_{k_1}^{(2,3,10)}T_{k_2}^{(-1,2,7)}T_{k_3}\\
&\quad +\frac{R}{22}\sum_{k_1+k_2+k_3+k_4=n\atop k_1,k_2,k_3,k_4\ge 0}\binom{n}{k_1,k_2,k_3,k_4}2^{k_1}T_{k_1}^{(2,3,10)}T_{k_2}T_{k_3}T_{k_4}\\
&\quad +\frac{S}{22}\sum_{k_1+\cdots+k_5=n\atop k_1,\dots,k_5\ge 0}\binom{n}{k_1,\dots,k_5}(-1)^{k_1}T_{k_1}^{(-1,2,7)}T_{k_2}T_{k_3}T_{k_4}\,,
\end{align*}
where $A=I+2L+2N+P+2Q+6R+4S-14$, $B=-2D-2N-5P-2Q-6R-12S+30$,
$C=-D-I-2L-2P-3Q-6R-7S+20$, $E=-I-2L-N-Q-3R-S+5$ and $H=-L-2N-P-Q-4R-3S+10$.
\label{tri5}
\end{theorem}

\noindent
{\it Remark.}
If $B=I=L=N=P=Q=R=S=0$, then by $A=-14$, $C=5$, $D=15$, $E=5$ and $H=10$, we have for $n\ge 0$,
\begin{align*}
&\sum_{k_1+\cdots+k_5=n\atop k_1,\dots,k_5\ge 1}\binom{n}{k_1,\dots,k_5}T_{k_1}\cdots T_{k_5}\\
&=\frac{1}{968}\left(-14\cdot 5^n T_n^{(5,6,15)}+5\sum_{k=0}^n\binom{n}{k}2^k T_k^{(2,3,10)}(1+2\cdot 3^{n-k}T_{n-k}^{(3,3,5)})\right)\\
&\quad +\frac{15}{44}\sum_{k_1+k_2+k_3=n\atop k_1,k_2,k_3\ge0}\binom{n}{k_1,k_2,k_3}T_{k_1}T_{k_2}
+\frac{5}{484}\sum_{k=0}^n\binom{n}{k}4^{n-k}T_{n-k}^{(2,14,21)}T_k\,.
\end{align*}


\section{More general results}

We shall consider the general case of Lemmata \ref{c^2}, \ref{c^3} and \ref{c^4}.
Similarly to the proof of Lemma \ref{c^2}, for Tribonacci-type numbers $s_{1,k}^{(n)}$, satisfying the recurrence relation $s_{1,k}^{(n)}=s_{1,k-1}^{(n)}+s_{1,k-2}^{(n)}+s_{1,k-3}^{(n)}$ ($k\ge 3$) with given initial values $s_{1,0}^{(n)}$, $s_{1,1}^{(n)}$ and $s_{1,2}^{(n)}$, we have the form
$$
d_1^{(n)}e^{\alpha x}+d_2^{(n)}e^{\beta x}+d_3^{(n)}e^{\gamma x}=\sum_{k=0}^\infty s_{1,k}^{(n)}\frac{x^k}{k!}\,.
$$

\begin{theorem}
For $n\ge 1$, we have
$$
c_1^n e^{\alpha x}+c_2^n e^{\beta x}+c_3^n e^{\gamma x}=\frac{1}{A_1^{(n)}}\sum_{k=0}^\infty T_{k}^{(s_{1,0}^{(n)},s_{1,1}^{(n)},s_{1,2}^{(n)})}\frac{x^k}{k!}\,,
$$
where $s_{1,0}^{(n)}$, $s_{1,1}^{(n)}$, $s_{1,2}^{(n)}$ and $A_1^{(n)}$ satisfy the following relations:
\begin{align}
&s_{1,0}^{(n)}=\pm{\rm lcm}(b_1,b_2), \quad
s_{1,1}^{(n)}=Bs_{1,0}^{(n)}, \quad
s_{1,2}^{(n)}=Cs_{1,0}^{(n)}, \notag\\
&A_1^{(n)}=\frac{A_1^{(n-1)}}{s_{1,1}^{(n-1)}}(3s_{1,2}^{(n)}-2s_{1,1}^{(n)}-s_{1,0}^{(n)})\,,
\label{s1A}
\end{align}
$b_1$, $b_2$, $B$ and $C$ are determined in the proof. 
\label{c^n}
\end{theorem}

\begin{proof}
By
$$
d_1^{(n)}=\frac{-s_{1,0}^{(n)}\beta\gamma+s_{1,1}^{(n)}(\beta+\gamma)-s_{1,2}^{(n)}}
{(\alpha-\beta)(\gamma-\alpha)},
$$
we can obtain
\begin{align*}
&\frac{-s_{1,0}^{(n)}\beta\gamma+s_{1,1}^{(n)}(\beta+\gamma)-s_{1,2}^{(n)}}{(\alpha-\beta)(\gamma-\alpha)}\\
&=\frac{A_1^{(n)}}{A_1^{(n-1)}}\frac{\alpha}{(\alpha-\beta)(\alpha-\gamma)}
\frac{-s_{1,0}^{(n-1)}\beta\gamma+s_{1,1}^{(n-1)}(\beta+\gamma)-s_{1,2}^{(n-1)}}{(\alpha-\beta)(\gamma-\alpha)}\,.
\end{align*}
After a few calculations, we get the relations (\ref{s1A}), 
where
\begin{align*}
&B=\frac{2s_{1,2}^{(n-1)}s_{1,1}^{(n-1)}+5s_{1,1}^{(n-1)}s_{1,0}^{(n-1)}+s_{1,1}^{(n-1)}s_{1,1}^{(n-1)}}
{5s_{1,2}^{(n-1)}s_{1,1}^{(n-1)}-4s_{1,1}^{(n-1)}s_{1,0}^{(n-1)}-3s_{1,1}^{(n-1)}s_{1,1}^{(n-1)}},\\
&C=\frac{D}{(3s_{1,2}^{(n-1)}-s_{1,1}^{(n-1)})(5s_{1,2}^{(n-1)}s_{1,1}^{(n-1)}
 -4s_{1,1}^{(n-1)}s_{1,0}^{(n-1)}-3s_{1,1}^{(n-1)}s_{1,1}^{(n-1)})},\\ 
&D=9s_{1,2}^{(n-1)}s_{1,2}^{(n-1)}s_{1,1}^{(n-1)}+6s_{1,2}^{(n-1)}s_{1,1}^{(n-1)}s_{1,0}^{(n-1)}
 +18s_{1,2}^{(n-1)}s_{1,1}^{(n-1)}s_{1,1}^{(n-1)}\quad\\
 &-2s_{1,1}^{(n-1)}s_{1,1}^{(n-1)}s_{1,0}^{(n-1)}
 -7s_{1,1}^{(n-1)}s_{1,1}^{(n-1)}s_{1,1}^{(n-1)},
\end{align*}
and
$$
B=\frac{a_1}{b_1},\quad
C=\frac{a_2}{b_2}, \quad \hbox{with}\quad \gcd(a_i,b_i)=1\,.
$$
We choose the symbol of $s_{1,0}^{(n)}$ such that for some $k_0$,
$T_{k}^{(s_{1,0}^{(n)},s_{1,1}^{(n)},s_{1,2}^{(n)})}$ is positive for all $k\ge k_0$.
\end{proof}

By Theorem \ref{c^n}, it is possible to obtain even higher-order convolution identities, but the forms seem to become more complicated.
For example, we can see that
\begin{eqnarray*}
c_1^6 e^{\alpha x}+c_2^6 e^{\beta x}+c_3^6 e^{\gamma x}=\frac{1}{2^4\cdot 11^3}\sum_{n=0}^\infty T_n^{(37,61,97)}\frac{x^n}{n!}\,,\\
c_1^7 e^{\alpha x}+c_2^7 e^{\beta x}+c_3^7 e^{\gamma x}=\frac{1}{2^4\cdot 11^3}\sum_{n=0}^\infty T_n^{(7,20,36)}\frac{x^n}{n!}\,,\\
c_1^8 e^{\alpha x}+c_2^8 e^{\beta x}+c_3^8 e^{\gamma x}=\frac{1}{2^5\cdot 11^4}\sum_{n=0}^\infty T_n^{(92,127,262)}\frac{x^n}{n!}\,,\\
c_1^9 e^{\alpha x}+c_2^9 e^{\beta x}+c_3^9 e^{\gamma x}=\frac{1}{2^6\cdot 11^4}\sum_{n=0}^\infty T_n^{(51,101,169)}\frac{x^n}{n!}\,,\\
c_1^{10} e^{\alpha x}+c_2^{10} e^{\beta x}+c_3^{10} e^{\gamma x}=\frac{1}{2^6\cdot 11^5}\sum_{n=0}^\infty T_n^{(169,347,658)}\frac{x^n}{n!}\,.
\end{eqnarray*}

\bigskip

Next, we shall consider a more general case of Lemma \ref{cc}.
Consider the Tribonacci-type numbers $s_{2,k}^{(n)}$, satisfying the recurrence relation $s_{2,k}^{(n)}=s_{2,k-1}^{(n)}+s_{2,k-2}^{(n)}+s_{2,k-3}^{(n)}$ ($k\ge 3$) with given initial values $s_{2,0}^{(n)}$, $s_{2,1}^{(n)}$ and $s_{2,2}^{(n)}$, we have the form
$$
r_1^{(n)}e^{\alpha x}+r_2^{(n)}e^{\beta x}+r_3^{(n)}e^{\gamma x}=\sum_{k=0}^\infty s_{2,k}^{(n)}\frac{x^k}{k!}\,,
$$
where  $r_1^{(n)}$, $r_2^{(n)}$ and $r_3^{(n)}$ are determined by solving the system of the equations
\begin{align*}
r_1^{(n)}+r_1^{(n)}+r_1^{(n)}&=s_{2,0}^{(n)}\,,\\
r_1^{(n)}\alpha+r_2^{(n)}\beta+r_3^{(n)}\gamma&=s_{2,1}^{(n)}\,,\\
r_1^{(n)}\alpha^2+r_2^{(n)}\beta^2+r_3^{(n)}\gamma^2&=s_{2,2}^{(n)}\,.
\end{align*}

\begin{theorem}
For $n\ge 1$, we have
$$
(c_2 c_3)^n e^{\alpha x}+(c_3 c_1)^n e^{\beta x}+(c_1 c_2)^n  e^{\gamma x}=\frac{1}{A_2^{(n)}}\sum_{k=0}^\infty T_{k}^{(s_{2,0}^{(n)},s_{2,1}^{(n)},s_{2,2}^{(n)})}\frac{x^k}{k!}\,.
$$
where $s_{2,0}^{(n)}$, $s_{2,1}^{(n)}$, $s_{2,2}^{(n)}$ and $A_2^{(n)}$ satisfy the following relations:
\begin{align}
&s_{2,0}^{(n)}=\pm{\rm lcm}(b_1,b_2), \quad
s_{2,1}^{(n)}=M s_{2,0}^{(n)}, \quad
s_{2,2}^{(n)}=N s_{2,0}^{(n)}, \notag\\
&A_2^{(n)}=\frac{A_2^{(n-1)}}{s_{2,2}^{(n-1)}-s_{2,1}^{(n-1)}-s_{2,0}^{(n-1)}}(-8s_{2,2}^{(n)}+18s_{2,1}^{(n)}+2s_{2,0}^{(n)})\,,
\label{s2A} 
\end{align}
$b_1$, $b_2$, $M$ and $N$ are determined in the proof. 
\label{cc^n}
\end{theorem}

\begin{proof}
By
\begin{align*}
r_1^{(n)}&=A_2^{(n)}c_2^n c_3^n,\\
r_1^{(n-1)}&=A_2^{(n-1)}c_2^{n-1} c_3^{n-1}\,,
\end{align*}
we can obtain the relations (\ref{s2A}), 
where
$$
M=\frac{a_1}{b_1},\quad
N=\frac{a_2}{b_2}, \quad \hbox{with}\quad \gcd(a_i,b_i)=1\,,
$$
and
\begin{align*}
&M=\frac{FA-CD}{BD-AE}, \quad
N=\frac{-1}{A}(BM+C),\\
&A=10s_{2,2}^{(n-1)}-10s_{2,1}^{(n-1)}-2s_{2,0}^{(n-1)}, \quad
B=-6s_{2,2}^{(n-1)}+6s_{2,1}^{(n-1)}-12s_{2,0}^{(n-1)},\\
&C=-8s_{2,2}^{(n-1)}+8s_{2,1}^{(n-1)}+6s_{2,0}^{(n-1)}, \quad
D=-6s_{2,2}^{(n-1)}+14s_{2,1}^{(n-1)}-2s_{2,0}^{(n-1)}, \\
&E=8s_{2,2}^{(n-1)}-26s_{2,1}^{(n-1)}+10s_{2,0}^{(n-1)}, \quad
F=18s_{2,2}^{(n-1)}-20s_{2,1}^{(n-1)}-16s_{2,0}^{(n-1)}.
\end{align*}

We choose the symbol of $s_{2,0}^{(n)}$ such that for some $k_0$, $T_{k}^{(s_{2,0}^{(n)},s_{2,1}^{(n)},s_{2,2}^{(n)})}$ is positive for all $k\ge k_0$.
\end{proof}
\medskip

As applications, we compute some values of $s_{2,0}^{(n)}$, $s_{2,1}^{(n)}$, $s_{2,2}^{(n)}$, $A_2^{(n)}$ for some $n$.

For $n=2$,  we have
\begin{align*}
&A=52, \quad B=-18, \quad C=-46, \quad D=-12, \quad E=-6, \quad F=102,\\
&M=9, \quad N=4, \quad s_{2,0}^{(2)}=1, \quad s_{2,1}^{(2)}=9, \quad s_{2,2}^{(2)}=4, \quad  A_2^{(2)}=22^2\,.
\end{align*}
Thus,
$$
c_2^2 c_3^2 e^{\alpha x}+c_3^2 c_1^2 e^{\beta x}+c_1^2 c_2^2  e^{\gamma x}=\frac{1}{22^2}\sum_{k=0}^\infty T_{k}^{(1,9,4)}\frac{x^k}{k!}\,.
$$

For $n=3$, we have
\begin{align*}
&A=-52, \quad B=18, \quad C=46, \quad D=100, \quad E=-192, \quad F=-124,\\
&M=\frac{-7}{31},\quad  N=\frac{25}{31}, \quad s_{2,0}^{3}=31, \quad s_{2,1}^{3}=-7, \quad s_{2,2}^{3}=25, \\
&A_2^{(3)}=2^4\cdot 11^3\,.
\end{align*}
Thus,
$$
(c_2 c_3)^3 e^{\alpha x}+(c_3 c_1)^3 e^{\beta x}+(c_1 c_2)^3 e^{\gamma x}=\frac{1}{2^4\cdot 11^3}\sum_{k=0}^\infty T_{k}^{(31,-7,25)}\frac{x^k}{k!}\,.
$$

For $n=4$, we have
\begin{align*}
&A=258, \quad B=-564, \quad C=-70, \quad D=-310, \quad E=692, \quad F=94,\\
&M=\frac{-29}{42}, \quad N=\frac{-26}{21}, \quad s_{2,0}^{4}=-42, \quad s_{2,1}^{4}=29, \quad s_{2,2}^{4}=52, \\
&A_2^{(4)}=2^5\cdot 11^4\,.
\end{align*}
Thus,
$$
c_2^4 c_3^4 e^{\alpha x}+c_3^4 c_1^4 e^{\beta x}+c_1^4 c_2^4  e^{\gamma x}=\frac{1}{2^5\cdot 11^4}\sum_{k=0}^\infty T_{k}^{(-42,29,52)}\frac{x^k}{k!}\,.
$$

For $n=5$, we have
\begin{align*}
&A=314, \quad B=366, \quad C=-436, \quad D=178, \quad E=-758, \quad F=1028,\\
&M=\frac{70}{53}, \quad N=\frac{-8}{53}, \quad s_{2,0}^{5}=53, \quad s_{2,1}^{5}=70, \quad s_{2,2}^{5}=-8, \\
&A_2^{(5)}=2^6\cdot 11^5\,.
\end{align*}
Thus,
$$
c_2^5 c_3^5 e^{\alpha x}+c_3^5 c_1^5 e^{\beta x}+c_1^5 c_2^5  e^{\gamma x}=\frac{1}{2^6\cdot 11^5}\sum_{k=0}^\infty T_{k}^{(53,70,-8)}\frac{x^k}{k!}\,.
$$

For $n=6$, we have
\begin{align*}
&A=-886, \quad B=-168, \quad C=942, \quad D=922, \quad E=-1354, \quad F=-2392,\\
&M=\frac{-217}{235}, \quad N=\frac{291}{235}, \quad s_{2,0}^{6}=235, \quad s_{2,1}^{6}=-217, \quad s_{2,2}^{6}=291, \\
&A_2^{(6)}=2^8\cdot 11^6\,.
\end{align*}
Thus,
$$
c_2^6 c_3^6 e^{\alpha x}+c_3^6 c_1^6 e^{\beta x}+c_1^6 c_2^6  e^{\gamma x}=\frac{1}{2^8\cdot 11^6}\sum_{k=0}^\infty T_{k}^{(235,-217,291)}\frac{x^k}{k!}\,.
$$

\noindent  
{\bf Conjecture.}  
By observing the following facts
\begin{align*}
&A_1^{(2)}=22=A_2^{(1)}, \quad A_1^{(4)}=22^2=A_2^{(2)}, \quad A_1^{(6)}=2^4\cdot 11^3=A_2^{(3)},\\
&A_1^{(8)}=2^5\cdot 11^4=A_2^{(4)}, \quad A_1^{(10)}=2^6\cdot 11^5=A_2^{(5)}\,,
\end{align*}
we conjecture that for $n\ge 1$, $A_1^{(2n)}=A_2^{(n)}$.
\bigskip

We can obtain more convolution identities for any fixed $n$, but we only give some of the results.  Similarly to Proposition \ref{tri2}, Theorems \ref{tri3}, \ref{tri4} and \ref{tri5}, we can give more general forms including symmetric formulae.  But for simplicity, we show some simple forms as seen in Remarks above. Their proofs are also similar and omitted.

\begin{theorem}
For $m\ge 0$ and $n\ge 1$, we have
\begin{align*}
&\frac{1}{(A_1^{(n)})^2}\sum_{k=0}^m\binom{m}{k}T_{k}^{(s_{1,0}^{(n)},s_{1,1}^{(n)},s_{1,2}^{(n)})} T_{m-k}^{(s_{1,0}^{(n)},s_{1,1}^{(n)},s_{1,2}^{(n)})}\\
&=\frac{1}{A_1^{(2n)}}2^mT_{m}^{(s_{1,0}^{(2n)},s_{1,1}^{(2n)},s_{1,2}^{(2n)})}
+\frac{2}{A_2^{(n)}}\sum_{k=0}^m\binom{m}{k}(-1)^k T_{k}^{(s_{2,0}^{(n)},s_{2,1}^{(n)},s_{2,2}^{(n)})}\,.
\end{align*}
\label{trin2}
\end{theorem}


\begin{theorem}
For $m\ge 0$ and $n\ge 1$, we have
\begin{align*}
&\frac{1}{(A_1^{(n)})^3}\sum_{k_1+k_2+k_3=m\atop k_1,k_2,k_3\ge 0}\binom{m}{k_1,k_2,k_3}
T_{k_1}^{(s_{1,0}^{(n)},s_{1,1}^{(n)},s_{1,2}^{(n)})}T_{k_2}^{(s_{1,0}^{(n)},s_{1,1}^{(n)},s_{1,2}^{(n)})}
T_{k_3}^{(s_{1,0}^{(n)},s_{1,1}^{(n)},s_{1,2}^{(n)})}\\
&=-\frac{2}{A_1^{(3n)}}3^m T_{m}^{(s_{1,0}^{(3n)},s_{1,1}^{(3n)},s_{1,2}^{(3n)})}+\frac{6}{44^n}\\
&\quad +\frac{3}{A_1^{(2n)}A_1^{(n)}}\sum_{k=0}^m\binom{m}{k}2^{k}T_{k}^{(s_{1,0}^{(2n)},s_{1,1}^{(2n)},s_{1,2}^{(2n)})}
T_{m-k}^{(s_{1,0}^{(n)},s_{1,1}^{(n)},s_{1,2}^{(n)})}\,.\\
\end{align*}
\label{trin3}
\end{theorem}

\begin{theorem}
For $m\ge 0$ and $n\ge 1$, we have
\begin{align*}
&\frac{1}{(A_1^{(n)})^4}\sum_{k_1+k_2+k_3+k_4=n\atop k_1,k_2,k_3,k_4\ge 0}\binom{n}{k_1,k_2,k_3,k_4}
T_{k_1}^{(s_{1,0}^{(n)},s_{1,1}^{(n)},s_{1,2}^{(n)})}\cdots T_{k_4}^{(s_{1,0}^{(n)},s_{1,1}^{(n)},s_{1,2}^{(n)})}\\
&=-\frac{6\cdot 4^m}{A_1^{(4n)}}T_{m}^{(s_{1,0}^{4n},s_{1,1}^{4n},s_{1,2}^{4n})}
+\frac{4}{A_1^{(3n)} A_1^{(n)}}\sum_{k=0}^m\binom{m}{k}3^{k}T_{k}^{(s_{1,0}^{3n},s_{1,1}^{3n},s_{1,2}^{3n})} T_{m-k}^{(s_{1,0}^{(n)},s_{1,1}^{(n)},s_{1,2}^{(n)})}\\
&\quad +\frac{3}{(A_1^{(2n)})^2}\sum_{k=0}^m\binom{m}{k}2^{m}T_{k}^{(s_{1,0}^{2n},s_{1,1}^{2n},s_{1,2}^{2n})}
T_{m-k}^{(s_{1,0}^{2n},s_{1,1}^{2n},s_{1,2}^{2n})}\\
&\quad +\frac{12}{44^nA_1^{(n)}}\sum_{k=0}^m\binom{m}{k}T_{k}^{(s_{1,0}^{(n)},s_{1,1}^{(n)},s_{1,2}^{(n)})}\,.
\end{align*}
\label{trin4}
\end{theorem}

\begin{theorem}
For $m\ge 0$ and $n\ge 1$, we have
\begin{align*}
&\frac{1}{(A_1^{(n)})^5}\sum_{k_1+\cdots+k_5=m\atop k_1,\dots,k_5\ge 0}\binom{m}{k_1,\dots,k_5}
T_{k_1}^{(s_{1,0}^{(n)},s_{1,1}^{(n)},s_{1,2}^{(n)})}\cdots T_{k_5}^{(s_{1,0}^{(n)},s_{1,1}^{(n)},s_{1,2}^{(n)})}\\
&=-\frac{14\cdot 5^m}{A_1^{(5n)}}T_{m}^{(s_{1,0}^{5n},s_{1,1}^{5n},s_{1,2}^{5n})}
+\frac{5}{44^n A_1^{(2n)}}\sum_{k=0}^m\binom{m}{k}2^{k}T_{k}^{(s_{1,0}^{2n},s_{1,1}^{2n},s_{1,2}^{2n})}\\
&\quad +\frac{15}{44^n(A_1^{(n)})^2}\sum_{k_1+k_2+k_3=m\atop k_1,k_2,k_3\ge0}\binom{m}{k_1,k_2,k_3}
T_{k_1}^{(s_{1,0}^{(n)},s_{1,1}^{(n)},s_{1,2}^{(n)})}T_{k_2}^{(s_{1,0}^{(n)},s_{1,1}^{(n)},s_{1,2}^{(n)})}\\
&\quad +\frac{5}{A_1^{(4n)} A_1^{(n)}}\sum_{k=0}^m\binom{m}{k}
4^{k}T_{k}^{(s_{1,0}^{4n},s_{1,1}^{4n},s_{1,2}^{4n})} T_{m-k}^{(s_{1,0}^{(n)},s_{1,1}^{(n)},s_{1,2}^{(n)})}\\
&\quad +\frac{10}{A_1^{(3n)} A_1^{(2n)}}\sum_{k=0}^m\binom{m}{k}
3^{k}2^{m-k}T_{k}^{(s_{1,0}^{3n},s_{1,1}^{3n},s_{1,2}^{3n})}T_{m-k}^{(s_{1,0}^{2n},s_{1,1}^{2n},s_{1,2}^{2n})}\,.\\
\end{align*}
\label{trin5}
\end{theorem}

\section{Some more interesting general expressions}

We shall give some more identities that the higher power can be expressed explicitly.

\begin{theorem}
For $n\ge 1$, we have
$$
(c_2 c_3+c_3 c_1+c_1 c_2)^n (e^{\alpha x}+e^{\beta x}+ e^{\gamma x})=\left(\frac{-1}{22}\right)^n\sum_{k=0}^\infty T_k^{(3,1,3)}\frac{x^k}{k!}\,.
$$
\label{(cc+cc+cc)^n}
\end{theorem}

If $n=1$, we have the following.

\begin{Cor}
We have
$$
(c_2 c_3+c_3 c_1+c_1 c_2) (e^{\alpha x}+e^{\beta x}+ e^{\gamma x})=\frac{-1}{22}\sum_{k=0}^\infty T_k^{(3,1,3)}\frac{x^k}{k!}\,.
$$
\label{cc+cc+cc}
\end{Cor}

\begin{proof}[Proof of Theorem \ref{(cc+cc+cc)^n}]
Similarly to the proof of Theorem \ref{cc^n}, we consider the form
$$
h_1^{(n)}e^{\alpha x}+h_2^{(n)}e^{\beta x}+h_3^{(n)}e^{\gamma x}=\sum_{n=0}^\infty s_{3,n}^{(n)}\frac{x^n}{n!}\,.
$$
By
\begin{align*}
&h_1^{(n)}=A_3^{(n)}(c_2 c_3+c_3 c_1+c_1 c_2)^n,\\
&h_1^{(n-1)}=A_3^{(n-1)}(c_2 c_3+c_3 c_1+c_1 c_2)^{n-1},
\end{align*}
we can obtain that
\begin{align*}
s_{3,1}^{(n)}=\frac{s_{3,1}^{(n-1)}}{s_{3,0}^{(n-1)}}s_{3,0}^{(n)},\quad
s_{3,2}^{(n)}=\frac{s_{3,2}^{(n-1)}}{s_{3,0}^{(n-1)}}s_{3,0}^{(n)},\quad
 A_3^{(n)}=-22A_3^{(n-1)}\frac{s_{3,0}^{(n)}}{s_{3,0}^{(n-1)}}\,.
\end{align*}
Thus,
\begin{align*}
s_{3,0}^{(n)}=3, \quad s_{3,1}^{(n)}=1, \quad s_{3,2}^{(n)}=3, \quad A_3^{(n)}=(-22)^n.
\end{align*}
\end{proof}

Similarly, we get the following.

\begin{theorem}
For $n\ge 1$, we have
\begin{multline*}
(c_2^2 c_3^2+c_3^2 c_1^2+c_1^2 c_2^2)^n (e^{\alpha x}+e^{\beta x}+ e^{\gamma x})\\
=\frac{1}{2^6\cdot 5\cdot 11^2 \cdot (22^2)^{n-1}}\sum_{k=0}^\infty T_k^{(242,82,245)}\frac{x^k}{k!}\,.
\end{multline*}
\label{(cc2+cc2+cc2)^n}
\end{theorem}

\begin{Cor}
We have
$$
(c_2^2 c_3^2+c_3^2 c_1^2+c_1^2 c_2^2) (e^{\alpha x}+e^{\beta x}+ e^{\gamma x})=\frac{1}{2^6\cdot 5\cdot 11^2}\sum_{k=0}^\infty T_k^{(242,82,245)}\frac{x^k}{k!}\,.
$$
\label{c^2c^2+c^2c^2+c^2c^2}
\end{Cor}

Finally, we can get the following.

\begin{theorem}
For $n\ge 1$, we have
$$
(c_2^2 +c_3^2)^n e^{\alpha x}+(c_3^2+ c_1^2)^n e^{\beta x}+ (c_1^2+ c_2^2)^n e^{\gamma x}=\frac{1}{A_5^{(n)}}\sum_{k=0}^\infty T_{k}^{(s_{5,0}^{(n)},s_{5,1}^{(n)},s_{5,2}^{(n)})}\frac{x^k}{k!}\,.
$$
\label{(c^2+c^2)^n}
\end{theorem}

\begin{Cor}
We have
$$
(c_2^2 +c_3^2) e^{\alpha x}+(c_3^2+ c_1^2)e^{\beta x}+ (c_1^2+ c_2^2)e^{\gamma x}=\frac{-1}{22}\sum_{k=0}^\infty T_k^{(-4,1,4)}\frac{x^k}{k!}\,.
$$
\label{c^2+c^2}
\end{Cor}

\begin{proof}[Proof of Theorem \ref{(c^2+c^2)^n}]
Consider the form
$$
\ell_1^{(n)}e^{\alpha x}+\ell_2^{(n)}e^{\beta x}+\ell_3^{(n)}e^{\gamma x}=\sum_{n=0}^\infty s_{5,n}^{(n)}\frac{x^n}{n!}\,.
$$
By $\ell_1=(c_2^2 +c_3^2)^n$,
we can obtain that
\begin{align*}
&s_{5,0}^{(n)}=\pm{\rm lcm}(b_1,b_2), \quad s_{5,1}^{(n)}=Ms_{5,0}^{(n)}, \quad s_{5,2}^{(n)}=Ns_{5,0}^{(n)}, \\
&A_5^{(n)}=\frac{44A_5^{(n-1)}s_{5,0}^{(n)}}{s_{5,2}^{(n-1)}-4s_{5,1}^{(n-1)}+3s_{5,0}^{(n-1)}}\,,
\end{align*}
where
$$
M=\frac{a_1}{b_1},\quad N=\frac{a_2}{b_2}\quad \hbox{with}\quad \gcd(a_i,b_i)=1,$$
and
$$
M=\frac{3s_{5,2}^{(n-1)}-4s_{5,1}^{(n-1)}-s_{5,0}^{(n-1)}}{s_{5,2}^{(n-1)}-4s_{5,1}^{(n-1)}+3s_{5,0}^{(n-1)}},\quad
N=\frac{-7s_{5,2}^{(n-1)}+10s_{5,1}^{(n-1)}+5s_{5,0}^{(n-1)}}{s_{5,2}^{(n-1)}-4s_{5,1}^{(n-1)}+3s_{5,0}^{(n-1)}}.
$$
We choose the symbol of $s_{5,0}^{(n)}$ such that for some $k_0$,
$T_{k}^{(s_{5,0}^{(n)},s_{5,1}^{(n)},s_{5,2}^{(n)})}$ is positive for all $k\ge k_0$.
\end{proof}

As applications, we compute some values of $s_{2,0}^{(n)}$, $s_{2,1}^{(n)}$, $s_{2,2}^{(n)}$, $A_2^{(n)}$ for some $n$.

For $n=2$,
$$
M=-1, \quad N=\frac{19}{6}, \quad s_{5,0}^{(2)}=6, \quad s_{5,1}^{(2)}=-6, \quad s_{5,2}^{(2)}=19, \quad A_5^{(2)}=22^2\,.
$$
For $n=3$,
\begin{align*}
&M=\frac{75}{61}, \quad N=\frac{-163}{61}, \quad s_{5,0}^{(3)}=-61, \quad s_{5,1}^{(3)}=-75, \quad s_{5,2}^{(3)}=163, \\
&A_5^{(3)}=-22^3\cdot 2\,.
\end{align*}
For $n=4$,
\begin{align*}
&M=\frac{85}{28}, \quad N=\frac{-549}{70}, \quad s_{5,0}^{(4)}=-140, \quad s_{5,1}^{(4)}=-425, \quad s_{5,2}^{(4)}=1098, \\
&A_5^{(4)}=22^4\cdot 2\,.
\end{align*}
For $n=5$,
\begin{align*}
&M=\frac{2567}{1189}, \quad N=\frac{-6318}{1189}, \quad s_{5,0}^{(5)}=-1189, \quad s_{5,1}^{(5)}=-2567, \\
&s_{5,2}^{(5)}=6318, \quad A_5^{(5)}=-22^5\cdot 2\,.
\end{align*}
For $n=6$,
\begin{align*}
&M=\frac{30411}{13019}, \quad N=\frac{-75841}{13019}, \quad s_{5,0}^{(6)}=-13019, \quad s_{5,1}^{(6)}=-30411, \\
&s_{5,2}^{(6)}=75841, \quad A_5^{(6)}=22^6\cdot 2^2\,.
\end{align*}

\section{Open problems}  

It has not been certain if there exists any simpler form or any explicit form about the higher-order of tribonacci convolution identities.  
It would be interesting to study similar convolution identities for numbers, satisfying more general recurrence relations.  For example, convolution identites of Tetranacci numbers $t_n$, satisfying the five term recurrence relation $t_n=t_{n-1}+t_{n-2}+t_{n-3}+t_{n-4}$, have been studied by Rusen Li (\cite{Li}).

\section{Acknowledgement}

The authors thank the anonymous referee for careful reading of the manuscript and useful suggestions.

\end{document}